\documentclass[smallextended,envcountsect]{my_svjour3}

\usepackage{amsmath,amscd,amsfonts,amssymb,enumerate}
\usepackage{graphicx}
\usepackage{subfig}

\usepackage{a4wide}

\usepackage[misc,geometry]{ifsym}

\begin{document}

\title{A Simple Accurate Method for Solving Fractional Variational and Optimal Control Problems}

\titlerunning{Solving Fractional Variational and Optimal Control Problems}

\author{Salman Jahanshahi \and Delfim F. M. Torres}

\authorrunning{S. Jahanshahi \and D. F. M. Torres}

\institute{S. Jahanshahi \at
Department of Mathematics, Science and Research Branch,\\
Islamic Azad University, Tehran, Iran\\
\email{s.jahanshahi@iausr.ac.ir}
\and
D. F. M. Torres \Letter \at
Center for Research and Development in Mathematics and Applications (CIDMA),\\
Department of Mathematics, University of Aveiro, 3810-193 Aveiro, Portugal\\
\email{delfim@ua.pt}
}

\date{Submitted: 31-Dec-2014 / Revised: 21-Oct-2015 and 17-Jan-2016 / Accepted: 24-Jan-2016.}

\maketitle


\begin{abstract}
We develop a simple and accurate method to solve fractional variational
and fractional optimal control problems with dependence on Caputo
and Riemann--Liouville operators. Using known formulas for computing
fractional derivatives of polynomials, we rewrite the fractional
functional dynamical optimization problem as a classical static
optimization problem. The method for classical optimal control problems
is called Ritz's method. Examples show that the proposed
approach is more accurate than recent methods available in the literature.


\end{abstract}


\section{Introduction}
\label{sec1}

In view of its history, fractional (non-integer order) calculus
is as old as the classical calculus \cite{MR2218073,MR1347689,MR3181071}.
Roughly speaking, there is just one definition of fractional integral operator,
which is the Riemann--Liouville fractional integral. There are, however, several
definitions of fractional differentiation, e.g., Caputo,
Riemann--Liouville, Hadamard, and Wily differentiation \cite{MR3224387,MR2960307}.
Each type of fractional derivative has its own properties,
which make richer the area of fractional calculus and enlarge
its range of applicability \cite{MR2768178,MR3188372}.

There are two recent research areas where fractional operators
have a particularly important role: the fractional calculus of variations
and the fractional theory of optimal control \cite{sal34,book:adv:FCV,MR2984893}.
A fractional variational problem is a dynamic optimization problem,
in which the objective functional, as well as its constraints, depends on derivatives or integrals
of fractional order, e.g., Caputo, Riemann--Liouville or Hadamard fractional operators.
This is a generalization of the classical theory, where derivatives and integrals can only appear
in integer orders. If at least one non-integer (fractional) term exists in its formulation,
then the problem is said to be a fractional variational problem
or a fractional optimal control problem.
The theory of the fractional calculus of variations
was introduced by Riewe in 1996, to deal with nonconservative systems in mechanics
\cite{sal1,sal2}. This subject has many applications in physics and engineering,
and provides more accurate models of physical phenomena. For this reason,
it is now under strong development: see \cite{sal5,sal3,sal4,MR3103208,MR3162654,MR3200762}
and the references therein. For a survey, see \cite{MR3221831}.

There are two main approaches to solve problems of the fractional calculus of variations or optimal control.
One involves solving fractional Euler--Lagrange equations or fractional Pontryagin-type conditions,
which is the indirect approach; the other involves addressing directly the problem,
without involving necessary optimality conditions, which is the direct approach.
The emphasis in the literature has been put on indirect methods \cite{book:adv:FCV,MR2984893,sal6,sal7}.
For direct methods, see \cite{sal34,sal3}. Furthermore,
Almeida et al. developed direct numerical methods based
on the idea of writing the fractional operators in power series,
and then approximating the fractional problems with classical ones \cite{MyID:294,sal8}.
In this paper, we use a different approach to solve fractional variational problems
and fractional optimal control problems, based on Ritz's direct method. The idea
is to restrict admissible functions to linear combinations
of a set of known basis functions.  We choose basis functions
in such a way that the approximated function satisfies
the given boundary conditions. Using the approximated function and its derivatives whenever needed,
we transform the functional into a multivariate function of unknown coefficients.
Recently, Dehghan et al. in \cite{sal9} used the Rayleigh--Ritz method, based
on Jacobi polynomials, to solve fractional optimal control problems. Several
illustrative examples show that our results are more accurate and more useful
than the ones introduced in \cite{sal8,sal9}.

The paper is organized as follows. In Section~\ref{sec2}, we present
some necessary preliminaries on fractional calculus. In Section~\ref{sec3},
we investigate Ritz's method for solving three kinds of fractional
variational problems. In Section~\ref{sec4}, we solve five examples
of fractional variational problems and three examples of fractional
optimal control problems, comparing our results with previous methods
available in the literature. The main conclusions are given
in Section~\ref{sec:conc}.


\section{Preliminaries and Notations About Fractional Calculus}
\label{sec2}

Riemann--Liouville fractional integrals are a generalization of the $n$ fold integral,
$n\in \mathbb{N}$, to real value numbers. Using the usual notation in the theory of fractional
calculus, we define the Riemann--Liouville fractional integrals as follows.

\begin{definition}[Fractional integrals]
The left and right Riemann--Liouville fractional integrals
of order $\alpha>0$ of a given function $f$ are defined by
\begin{equation*}
_aI_x^\alpha f(x):=\frac{1}{\Gamma(\alpha)}
\int_a^x(x-t)^{\alpha-1}f(t)dt
\end{equation*}
and
\begin{equation*}
_xI_b^{\alpha}f(x):=\frac{1}{\Gamma(\alpha)}\int_x^b(t-x)^{\alpha-1}f(t)dt,
\end{equation*}
respectively, where $\Gamma$ is Euler's gamma function, that is,
\begin{equation*}
\Gamma(x):=\int_0^\infty t^{x-1}e^{-t}dt,
\end{equation*}
and $a<x<b$.
\end{definition}

The left Riemann--Liouville fractional operator has the following properties:
\begin{equation*}
{_aI_x^\alpha} {_aI_x^\beta}
= {_aI_x}^{\alpha+\beta},
\quad {_aI_x^\alpha} {_aI_x^\beta}
= {_aI_x^\beta} {_aI_x^\alpha},
\end{equation*}
\begin{equation}
\label{eq11}
_aI_x^\alpha (x-a)^n=\frac{\Gamma(n+1)}{\Gamma(\alpha+n+1)}(x-a)^{\alpha+n}, \quad x>a,
\end{equation}
where $\alpha,\beta\geq 0$ and $n\in \mathbb{N}_0=\{0,1,\ldots\}$. Similar relations
hold for the right Riemann--Liouville fractional operator. Now, using the definition
of fractional integral, we define two kinds of fractional derivatives.

\begin{definition}[Riemann--Liouville derivatives]
Let $\alpha>0$ with $n-1<\alpha\leq n$, $n\in\mathbb{N}$.
The left and right Riemann--Liouville fractional derivatives
of order $\alpha>0$ of a given function $f$ are defined by
\begin{equation*}
_aD_x^{\alpha}f(x):=D^n(_aI_x^{n-\alpha})f(x)
=\frac{1}{\Gamma(n-\alpha)}
\frac{d^n}{dx^n}\int_a^x(x-t)^{n-\alpha-1}f(t)dt
\end{equation*}
and
\begin{equation*}
_xD_b^\alpha f(x):=(-1)^nD^n(_xI_b^{n-\alpha})f(x)
=\frac{(-1)^n}{\Gamma(n-\alpha)}
\frac{d^n}{dx^n}\int_x^b(t-x)^{n-\alpha-1}f(t)dt,
\end{equation*}
respectively.
\end{definition}

The following relations hold for the left Riemann--Liouville fractional derivative:
\begin{equation*}
{_aD_x^\alpha} {_aD_x^\beta}
= {_aD_x}^{\alpha+\beta},
\quad {_aD_x^\alpha} {_aD_x^\beta}
= {_aD_x^\beta} {_aD_x^\alpha},
\end{equation*}
$$
_aD_x^\alpha k= \frac{k (x-a)^\alpha}{\Gamma(1-\alpha)},
$$
\begin{equation}
\label{eq12}
_aD_x^\alpha (x-a)^n=\frac{\Gamma(n+1)}{\Gamma(n-\alpha+1)}(x-a)^{n-\alpha}, \quad x>a,
\end{equation}
where $\alpha,\beta\geq 0$, $k$ is a constant and $n\in \mathbb{N}_0$.
Similar relations hold for the right Riemann--Liouville fractional derivative.

Another type of fractional derivative, which also uses the Riemann--Liouville
fractional integral in its definition, was proposed by Caputo \cite{sal16}.

\begin{definition}[Caputo derivatives]
Let $\alpha>0$ with $n-1<\alpha\leq n$, $n\in\mathbb{N}$.
The left and right Caputo fractional derivatives
of order $\alpha>0$ of a given function $f$ are defined by
\begin{equation*}
_a^CD_x^\alpha f(x):=(_aI_x^{n-\alpha})D^n f(x)
=\frac{1}{\Gamma(n-\alpha)}\int_a^x(x-t)^{n-\alpha-1}f^{(n)}(t)dt
\end{equation*}
and
\begin{equation*}
_x^CD_b^\alpha f(x):=(-1)^n(_xI_b^{n-\alpha})D^n f(x)
=\frac{(-1)^n}{\Gamma(n-\alpha)}\int_x^b(t-x)^{n-\alpha-1}f^{(n)}(t)dt,
\end{equation*}
respectively.
\end{definition}

The following relations hold for the left Caputo fractional derivative:
\begin{equation*}
{_a^CD_x^\alpha} \, {_a^CD_x^\beta}
= {{_a^CD_x^{\alpha+\beta}}},
\quad {_a^CD_x^\alpha} \, {_a^CD_x^\beta}
= {_a^CD_x^\beta} \, {_a^CD_x^\alpha},
\end{equation*}
$$
_a^CD_x^\alpha k = 0,
$$
\begin{equation}
\label{eq13}
_a^CD_x^\alpha (x-a)^n=\left\{
\begin{array}{ll}
\frac{\Gamma(n+1)}{\Gamma(n-\alpha+1)}(x-a)^{n-\alpha},
&  n\geq [\alpha],\\
0, & n< [\alpha],
\end{array} \right.
\end{equation}
where $\alpha,\beta\geq 0$, $k$ is a constant, $x>a$,
$n\in \mathbb{N}_0$, and $[\cdot]$ is the ceiling function,
that is, $[\alpha]$ is the smallest integer greater than or equal to $\alpha$. Similar relations
hold for the right Caputo fractional derivative. Moreover, the following relations between Caputo
and Riemann--Liouville fractional derivatives hold:
\begin{equation*}
_a^CD_x^\alpha f(x)={_aD_x^\alpha} f(x)
-\sum_{k=0}^{n-1}\frac{f^{(k)}(a)}{\Gamma(k-\alpha+1)}(x-a)^{k-\alpha}
\end{equation*}
and
\begin{equation*}
_x^CD_b^\alpha f(x)={_xD_b^\alpha} f(x)
-\sum_{k=0}^{n-1}\frac{f^{(k)}(b)}{\Gamma(k-\alpha+1)}(b-x)^{k-\alpha}.
\end{equation*}
Therefore, if $f\in C^n[a,b]$ and $f^{(k)}(a)=0$, $k=0,1,\ldots,n-1$, then
$_a^CD_x^\alpha f= {_aD_x^\alpha} f$;
if $f^{(k)}(b)=0$, $k=0,1,\ldots,n-1$, then
$_x^CD_b^\alpha f={_xD_b^\alpha} f$.


\section{Main Results}
\label{sec3}

Consider the following fractional variational problem:
find $y\in C^n[a,b]$ in such a way to minimize or maximize the functional
\begin{equation}
\label{eq3.1}
J\{y\}=\int_a^b L(t,y(t),D^\alpha y(t))dt, \quad n-1\leq \alpha<n,
\end{equation}
subject to boundary conditions
\begin{equation}
\label{eq3.2}
y^{(k)}(a)=u_{k,a},\quad y^{(k)}(b)=u_{k,b}, \quad k=0,1,\ldots,n-1,
\end{equation}
where $L$ is the Lagrangian, assumed to be continuous with respect to all its arguments,
$D^\alpha$ is a fractional operator (left or right Riemann--Liouville fractional
integral or derivative or left or right Caputo fractional derivative), and $u_{k,a}$
and $u_{k,b}$, $k=0,1,\ldots,n-1$, are given constants. To solve problem
\eqref{eq3.1}--\eqref{eq3.2} with our method, we need to recall the following
classical theorems from approximation theory.

\begin{theorem}[Stone--Weierstrass theorem (1937)]
\label{thm:sw}
Let $K$ be a compact metric space and $A\subset C(K, R)$ a unital
sub-algebra, which separates points of $K$. Then, $A$ is dense in $C(K, R)$.
\end{theorem}

\begin{theorem}[Weierstrass approximation theorem (1885)]
\label{thm:wa}
If $f \in C([a, b], \mathbb{R})$, then there is a sequence of polynomials $P_n(x)$
that converges uniformly to $f(x)$ on $[a, b]$.
\end{theorem}

\begin{theorem}
\label{theorem3}
Let $P_N(x):=\sum_{i=0}^N c_i (x-a)^i$ be a polynomial. Then,
\begin{equation}
\label{eq14}
_aI_x^\alpha P_N(x)=\sum_{i=0}^N c_i
\frac{\Gamma(i+1)}{\Gamma(\alpha+i+1)}(x-a)^{\alpha+i},
\end{equation}
\begin{equation}
\label{eq15}
_aD_x^\alpha P_N(x)=\sum_{i=0}^N c_i
\frac{\Gamma(i+1)}{\Gamma(i-\alpha+1)}(x-a)^{i-\alpha},
\end{equation}
\begin{equation}
\label{eq16}
_a^CD_x^\alpha P_N(x)
=\sum_{i = [\alpha]}^{N}\frac{\Gamma(i+1)}{\Gamma(i-\alpha+1)}(x-a)^{i-\alpha},
\end{equation}
for $x>a$.
\end{theorem}

\begin{proof}
Follows from relations \eqref{eq11}, \eqref{eq12} and \eqref{eq13}
and the linearity property of the fractional operators.
\hfill $\qed$
\end{proof}

For solving fractional variational problems involving right fractional operators,
we use the following theorem.

\begin{theorem}
\label{theorem4}
Let $P_N(x):=\sum_{i=0}^{N} c_i (b-x)^i$ be a polynomial. Then,
\begin{equation}
\label{eq30}
_xI_b^\alpha P_N(x)=\sum_{i=0}^N c_i
\frac{\Gamma(i+1)}{\Gamma(\alpha+i+1)}(b-x)^{\alpha+i},
\end{equation}
\begin{equation}
\label{eq31}
_xD_b^\alpha P_N(x)=\sum_{i=0}^N c_i
\frac{\Gamma(i+1)}{\Gamma(i-\alpha+1)}(b-x)^{i-\alpha},
\end{equation}
\begin{equation}
\label{eq32}
_x^CD_b^\alpha P_N(x)
=\sum_{i=[\alpha]}^{N}\frac{\Gamma(i+1)}{\Gamma(i-\alpha+1)}(b-x)^{i-\alpha},
\end{equation}
for $x<b$.
\end{theorem}

\begin{proof}
Similar to the proof of Theorem~\ref{theorem3}. \hfill $\qed$
\end{proof}

Without any loss of generality, from now on we consider $a = 0$, $b = 1$
and $t\in [0, 1]$ in the fractional variational problem \eqref{eq3.1}--\eqref{eq3.2}.


\subsection{Fractional Variational Problems Involving Left Operators}
\label{sec:MR:FVP:LO}

Consider the fractional variational functional \eqref{eq3.1} involving
a left fractional operator, subject to boundary conditions \eqref{eq3.2}.
To find the function $y(x)$ that solves problem
\eqref{eq3.1}--\eqref{eq3.2}, we put
\begin{equation}
\label{eq:yN:L}
y(x)\simeq y_N(x)=\sum_{i=0}^N c_i x^i.
\end{equation}
Then, by substituting \eqref{eq:yN:L}
and relations \eqref{eq14}--\eqref{eq16} into \eqref{eq3.1}, we obtain
\begin{equation}
\label{eq18}
J[y]=J[c_0,c_1,\ldots ,c_N]=\int_0^1 L(x,y_N(x),D^\alpha y_N(x))dx
\end{equation}
subject to boundary conditions
\begin{equation}
\label{eq19}
y_N^{(k)}(0)=u_{k,0}, \quad y_N^{(k)}(1)=u_{k,1},
\quad k=0,1,\ldots, n-1,
\end{equation}
which is an algebraic function of unknowns $c_i$, $i=0,1,\ldots, N$.
To optimize the algebraic function $J$, we act as follows.
We should find $c_0,c_1,\ldots,c_N$, such that $y_N$ satisfies
boundary conditions \eqref{eq19}. This means that
the following relations must be satisfied:
\begin{equation}
\label{eq20}
y_N^{(k)}(0)=u_{k,0}, \quad y_N^{(k)}(1)=u_{k,1},
\quad k=0,1,\ldots,n-1.
\end{equation}
Using the well known Hermite interpolation and the relations \eqref{eq20},
we calculate $c_0,c_1,\ldots,c_n$. For obtaining the values of $c_{n+1},c_{n+2},\ldots,c_N$,
firstly we use the $N+1$ point Gauss--Legendre quadrature rule, which is exact
for every polynomial of degree up to $2N+2$. Secondly, we calculate the exact value
of the integral in the right-hand side of \eqref{eq18}. Then, according
to differential calculus, we must solve the following system of equations:
\begin{equation}
\label{eq21}
\frac{\partial J}{\partial c_j}=0,
\quad j=n+1,n+2,\ldots,N,
\end{equation}
which, depending on the form of $L$, is a linear or nonlinear system of equations.
Furthermore, we choose the value of $N$ such that $|y_{N+1}(x)-y_N(x)|\cong 0(x)$,
where $0(x)$ is the null polynomial.


\subsection{Fractional Variational Problems Involving Right Operators}

Now consider the fractional variational functional \eqref{eq3.1}
involving a right fractional operator, subject to boundary conditions
\eqref{eq3.2}. To find function $y(x)$ that solves
problem \eqref{eq3.1}--\eqref{eq3.2}, we put
\begin{equation}
\label{eq:yN:R}
y(x)\simeq y_N(x)=\sum_{i=0}^N c_i (1-x)^i.
\end{equation}
Then, substituting \eqref{eq:yN:R} and relations
\eqref{eq30}--\eqref{eq32} into \eqref{eq3.1}, we obtain
\begin{equation}
\label{eq:alg:R}
J[y]=J[c_0,c_1,\ldots ,c_N]=\int_0^1 L(x,y_N(x),D^\alpha y_N(x))dx
\end{equation}
subject to boundary conditions
\begin{equation*}
y_N^{(k)}(0)=u_{k,0}, \quad y_N^{(k)}(1)=u_{k,1}, \quad k=0,1,\ldots, n-1,
\end{equation*}
which is an algebraic function of unknowns $c_i$, $i=0,1,\ldots, N$.
To optimize the algebraic function \eqref{eq:alg:R},
we act as explained in Section~\ref{sec:MR:FVP:LO}.


\subsection{Fractional Optimal Control Problems}

A fractional optimal control problem requires finding
a control function $u(t)$ and the corresponding
state trajectory $x(t)$, that minimizes (or maximizes)
a given functional
\begin{equation}
\label{eq9}
J\{x,u\}=\int_a^b L\left(t,x(t),u(t)\right) dt
\end{equation}
subject to a fractional dynamical control system
\begin{equation}
\label{eq7}
D^\alpha x(t)=f(t,x(t),u(t))
\end{equation}
and boundary conditions
\begin{equation}
\label{eq:foc:bc}
x^{(k)}(a)=u_{k,a},\quad x^{(k)}(b)=u_{k,b},
\quad k=0,1,\ldots,n-1,
\end{equation}
where $D^\alpha$ is a fractional operator, $\alpha$ is a positive real number,
and $f$ and $L$ are two known functions. For more details, see
\cite{MyID:294,sal14,MR2433010,MR2386201,sal15} and the references therein.

Here we restrict our attention to those fractional optimal control problems
\eqref{eq9}--\eqref{eq:foc:bc}, for which one can solve \eqref{eq7}
with respect to $u$ and write
\begin{equation}
\label{eq8}
u(t)=g\left(t,x(t),D^\alpha x(t)\right).
\end{equation}
Then, by substituting \eqref{eq8} into \eqref{eq9},
we obtain the following fractional variational problem (FVP):
find function $x(t)$ that extremizes the functional
\begin{equation}
\label{3q10}
J\{x\}=\int_a^b L\left(t,x(t),g(t,x(t),D^\alpha x(t))\right) dt
\end{equation}
subject to boundary conditions
\begin{equation}
\label{3q10:bc}
x^{(k)}(a)=u_{k,a},\quad x^{(k)}(b)=u_{k,b},
\quad k=0,1,\ldots,n-1.
\end{equation}
We solve the FVP \eqref{3q10}--\eqref{3q10:bc} by
using the method illustrated in previous sections,
and then we find control function $u(t)$ by using \eqref{eq8}.

\begin{remark}
Similarly to classical Ritz's method, the trial functions are selected
to meet boundary conditions (and any other constraints). The exact solutions
are not known; and the trial functions are parametrized by adjustable
coefficients, which are varied to find the best approximation
for the basis functions used. The choice of the basis functions
depends on the solution space. Since our problems
involve finding $C^n$ solutions, we choose the basis functions
to be polynomials. If the solution space is another one, like
the $L_p$ space or the space of harmonic functions, then
we should choose some basis functions like Fourier or wavelet basis.
\end{remark}

\begin{remark}
For choosing the number $N$ one needs to decide on the required accuracy.
One can stop when the difference between two consecutive approximations
$y_N$ and $y_{N-1}$ or their respective functional values is smaller
than a desired tolerance, that is, when
$\left\|y_N - y_{N-1}\right\| \leq \varepsilon$ or when
$\left|J[y_N] - J[y_{N-1}]\right| \leq \varepsilon$ for some given $\varepsilon$.
Note that for a minimization (maximization) problem, the functional $J$ is always
a non-increasing (non-decreasing) function of $N$:
$J[y_N] \leq J[y_{N-1}]$ ($J[y_N] \geq J[y_{N-1}]$). Moreover, if $y$
is the (unknown) exact solution of the fractional variational problem
at hand, then $y_N$ converges uniformly to $y$
(Theorems~\ref{thm:sw} and \ref{thm:wa}).
\end{remark}

\begin{remark}
Throughout the paper, $N$ is greater than $n$.
\end{remark}

In the next section, we solve five fractional problems
of the calculus of variations and three fractional
optimal control problems. Six of the eight problems
involve left fractional operators, while the other two
involve right operators.


\section{Illustrative Examples}
\label{sec4}

We begin by solving two problems of the calculus of variations
involving a left Riemann--Liouville fractional derivative.
These problems were recently investigated by Almeida et al. in \cite{sal34}.
We also solve three examples involving left and right Caputo fractional derivatives,
which were recently considered by Dehghan et al. in \cite{sal9}, and we compare the results.
Finally, we solve three fractional optimal control problems that
were investigated before in \cite{MyID:294,sal15,sal17}.
For all examples, the error between the exact solution $y$
and the approximate solution $y_N$, found using our method,
is computed as follows:
\begin{equation}
\label{error}
Error\{y,y_N\}=\int_0^1(y(x)-y_N(x))^2 dx.
\end{equation}
The results show that our method is very simple but very accurate,
providing better results than those found in the literature.

\begin{example} \cite{sal34,sal8}
\label{example3.1}
Let $\alpha \in ]0, 1[$. Consider the following FVP:
\begin{equation}
\label{eq3.10}
\text{minimize} \quad J\{y\}
=\int_0^1 \left(_0D_x^\alpha y(x)-(y')^2(x)\right) dx,
\end{equation}
subject to
\begin{equation}
\label{eq3.10:bc}
y(0)=0, \quad y(1)=1.
\end{equation}
The exact solution to problem \eqref{eq3.10}--\eqref{eq3.10:bc} is
$$
y(x)=-\frac{1}{2\Gamma(3-\alpha)}(1-x)^{2-\alpha}
+\left(1-\frac{1}{2\Gamma(3-\alpha)}\right)x
+\frac{1}{2\Gamma(3-\alpha)}
$$
\cite{sal34,sal8}. With the given boundary conditions
\eqref{eq3.10:bc}, we consider
$$
y_N(x)=\left(1-\sum_{i=2}^N c_i\right) x
+\sum_{i=2}^N c_i x^i.
$$
Using our method, we calculate the $c_i$'s by solving the system of equations \eqref{eq21}.
For $\alpha=0.5$ and $N=2$, we have
$$
y(x)=0.376126 - 0.376126 (1 - x)^{1.5} + 0.623874 x,
$$
$$
y_2(x)=1.22568 x - 0.225676 x^2.
$$
Figure~\ref{Fig:1} plots the result for  $\alpha=0.5$ and $N=2$.
The errors computed with \eqref{error} for different values of $\alpha$ and $N$
are presented in Table~\ref{Table3.3}.
\begin{figure}[!htb]
\begin{center}
\includegraphics[scale=1.0]{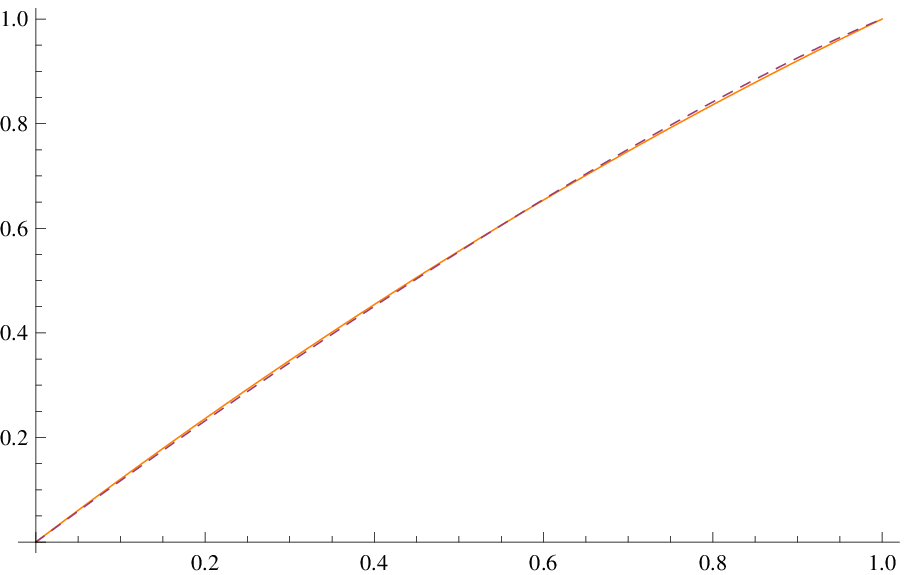}
\begin{minipage}{0.6\textwidth}
\caption{{\small Exact solution to problem \eqref{eq3.10}--\eqref{eq3.10:bc}
of Example~\ref{example3.1} with $\alpha=0.5$ (the solid line)
versus our approximate solution with $N=2$ (the dashed line).}}
\label{Fig:1}
\end{minipage}
\end{center}
\end{figure}
\begin{table}[ht]
\begin{center}
\begin{minipage}{0.58\textwidth}
\caption{{\small Errors obtained for problem \eqref{eq3.10}--\eqref{eq3.10:bc}
of Example~\ref{example3.1} when computed with \eqref{error}
for $\alpha=0.25,0.5,0.75$ and $N=3,4,5,6$.}}\label{Table3.3}
\end{minipage}
\begin{tabular}{|c|c|c|c|c|c|c|}
\hline
$\alpha$ & $N=3$&$N=4$&$N=5$&$N=6$\\\hline
$0.25$ & $1.7\times10^{-7}$&$1.9\times10^{-8}$&$3.6\times10^{-9}$&$9.1\times10^{-10}$\\
$0.5$ & $9.7\times10^{-6}$&$1.5\times10^{-7}$&$3.5\times10^{-8}$&$1.1\times10^{-8}$ \\
$0.75$ & $1.6\times10^{-6}$&$3.4\times10^{-7}$&$9.9\times10^{-8}$&$3.5\times10^{-8}$\\\hline
\end{tabular}
\end{center}
\end{table}
Our approximate results are better than
the ones presented in \cite{sal34,sal8}.
\end{example}


\begin{example} \cite{sal34}
\label{example3.2}
Let $\alpha=0.5$ and consider the minimization problem
\begin{equation}
\label{eq:J:ex3.2}
\text{minimize} \quad J\{y\}=\int_0^1 \left(_0D_x^\alpha y(x)-\frac{16\Gamma(6)}{\Gamma(5.5)}x^{4.5}
+\frac{20\Gamma(4)}{\Gamma(3.5)}x^{2.5}-\frac{5}{\Gamma(2.5)}x^{0.5}\right)^4 dx,
\end{equation}
subject to the boundary conditions
\begin{equation}
\label{eq:J:ex3.2:bc}
y(0)=0, \quad y(1)=1.
\end{equation}
The minimizer to the FVP \eqref{eq:J:ex3.2}--\eqref{eq:J:ex3.2:bc} is given by
$y(x)=16x^5-20x^3+5x$. For $N=5$, we obtain
$$
y_5(x)=5x + (2.3995\times 10^{-10}) x^2 {- 20} x^3 + (7.70126\times 10^{-10})x^4 + 16x^5.
$$
Figure~\ref{Fig:3} shows that our results are more accurate than the results presented in \cite{sal34}.
\begin{figure}[!htb]
\begin{center}
\subfloat[$N=4$]{\label{label:fig11}
\includegraphics[scale=0.77]{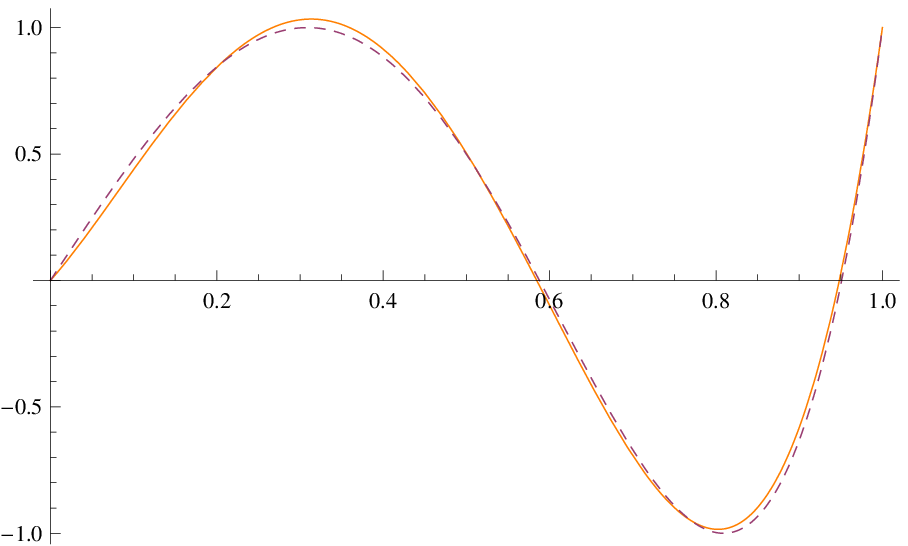}}
\qquad
\subfloat[$N=5$]{\label{label:fig21}
\includegraphics[scale=0.77]{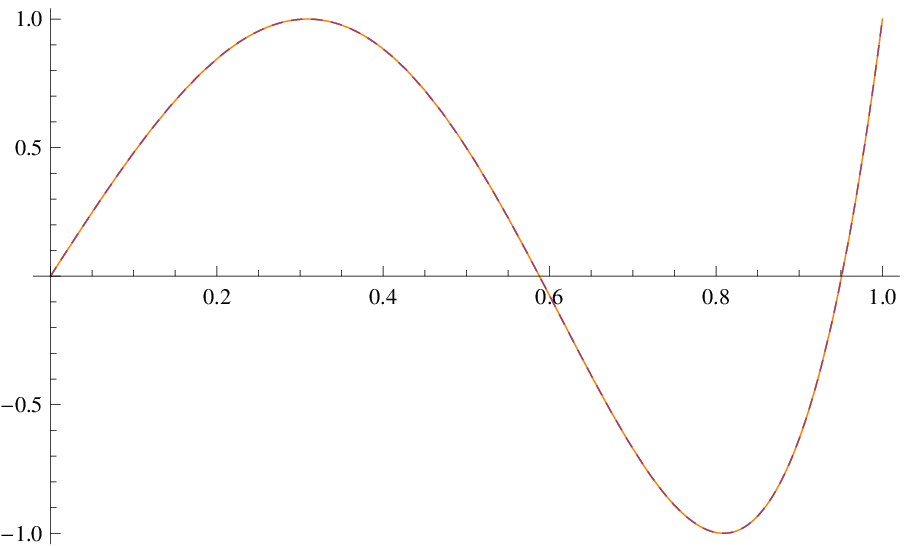}}\\
\begin{minipage}{0.95\textwidth}
\caption{{\small Exact solution to problem \eqref{eq:J:ex3.2}--\eqref{eq:J:ex3.2:bc}
of Example~\ref{example3.2} (the solid line) versus our approximate solutions
with $N=4,5$ (the dashed line).}}
\label{Fig:3}
\end{minipage}
\end{center}
\end{figure}
\end{example}


Below we solve three problems of the calculus of variations,
which were recently solved by Dehghan et al. in \cite{sal9}. Our results
show that our method is also more accurate than the method introduced in \cite{sal9}.

\begin{example} \cite{sal9}
\label{example3.3}
Consider the following FVP:
\begin{equation}
\label{eq:J:ex3.3}
\text{minimize}\quad J\{y\}=\frac{1}{2}\int_0^1(_0^CD_x^\alpha y(x)-f(x))^2dx,
\quad 0<\alpha <1,
\end{equation}
where $f(x)$ is given by
$$
f(x)=\frac{\Gamma(\beta +1)}{\Gamma(1+\beta -\alpha)}x^{\beta -\alpha},
$$
subject to the boundary conditions
\begin{equation}
\label{eq:J:ex3.3:bc}
y(0)=0, \quad y(1)=1.
\end{equation}
The exact solution to this problem is $y(x)=x^\beta$ \cite{sal9}.
In Figure~\ref{Fig:4}, the exact solution for $\alpha=0.5$ and $\beta=2.5$
versus our numerical solution for $N=3$ is plotted. Note that
for $\beta=k$, $k \in \mathbb{N}$, the error is equal to zero when $N\geq k$.
\begin{figure}[!htb]
\begin{center}
\includegraphics[scale=1.0]{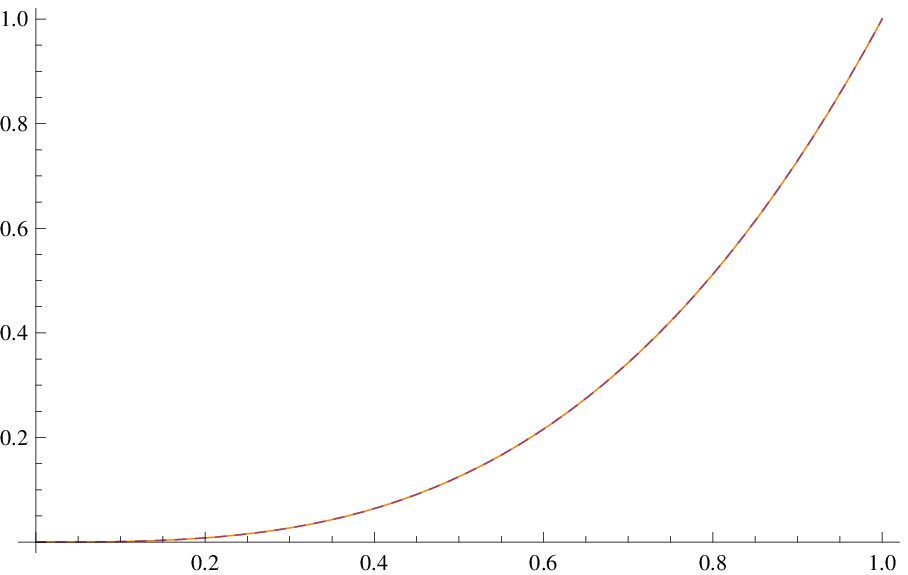}
\begin{minipage}{0.6\textwidth}
\caption{{\small Exact solution to problem \eqref{eq:J:ex3.3}--\eqref{eq:J:ex3.3:bc}
of Example~\ref{example3.3} with $\alpha=0.5$ and $\beta=2.5$ (the solid line)
versus our approximate solution with $N=3$ (the dashed line).}}
\label{Fig:4}
\end{minipage}
\end{center}
\end{figure}
\end{example}


\begin{example} \cite{sal9}
\label{3x}
Consider the following FVP, depending on a right Caputo fractional derivative:
\begin{equation}
\label{eq:ex:3x}
\text{minimize}\quad J\{y\}=\frac{1}{2}\int_0^1(_x^CD_1^\alpha y(x)-1)^2dx,
\quad 0<\alpha <1,
\end{equation}
subject to boundary conditions
\begin{equation}
\label{eq:ex:3x:bc}
y(0)=1+\frac{1}{\Gamma(1+\alpha)}, \quad y(1)=1.
\end{equation}
The exact solution to \eqref{eq:ex:3x}--\eqref{eq:ex:3x:bc} is given by
\begin{equation*}
y(x)=1+\frac{(1-x)^\alpha}{\Gamma(1+\alpha)}.
\end{equation*}
An approximate solution obtained with our method is shown in Figure~\ref{Fig:5}.
See Table~\ref{Table3.4} for the error of our approximations.
\begin{figure}[!htb]
\begin{center}
\includegraphics[scale=1.0]{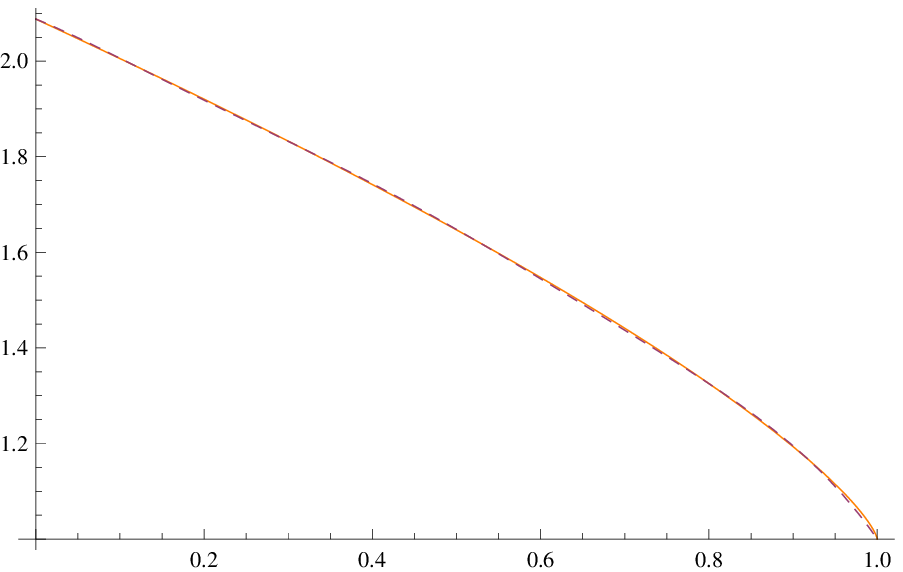}
\begin{minipage}{0.6\textwidth}
\caption{{\small Exact solution to problem \eqref{eq:ex:3x}--\eqref{eq:ex:3x:bc}
of Example~\ref{3x} with $\alpha=0.75$ (the solid line)
versus our approximate solution with $N=6$ (the dashed line).}}
\label{Fig:5}
\end{minipage}
\end{center}
\end{figure}
\begin{table}[ht]
\begin{center}
\begin{minipage}{0.65\textwidth}
\caption{{\small Errors obtained for problem \eqref{eq:ex:3x}--\eqref{eq:ex:3x:bc}
of Example~\ref{3x}, computed with \eqref{error} for $\alpha=0.75$
and $N=2,3,4,5,6$.}}\label{Table3.4}
\end{minipage}
\begin{tabular}{|c|c|c|c|c|c|c|}
\hline
$\alpha$ &$N=2$& $N=3$&$N=4$&$N=5$&$N=6$\\\hline
$0.75$ & $5.7\times10^{-4}$&$1.2\times10^{-4}$&$3.6\times10^{-5}$&$1.6\times10^{-5}$&$7.2\times10^{-6}$\\\hline
\end{tabular}
\end{center}
\end{table}
\end{example}


\begin{example} \cite{sal9}
\label{example4}
As a second example involving a fractional right operator,
consider the following FVP with $0<\alpha <1$:
\begin{equation}
\label{eq:prb:ex4}
\text{minimize}\quad J\{y\}=\int_0^1\left(_x^CD_1^\alpha y(x)+y(x)-(1-x)^\beta
-\frac{\Gamma(\beta +1)(1-x)^{\beta -\alpha}}{\Gamma(\beta -\alpha +1)}\right)^2 dx,
\end{equation}
subject to boundary conditions
\begin{equation}
\label{eq:prb:ex4:bc}
y(0)=1, \quad y(1)=0.
\end{equation}
In this case the exact solution is $y(x)=(1-x)^\beta$ \cite{sal9}. Figure~\ref{fig6}
shows the exact solution for $\alpha=0.39$ and $\beta=3$ versus our numerical solution with $N=3$.
The errors \eqref{error} of our approximations, for different values of $(\alpha,\beta)$,
$\beta\neq3$, are listed in Table~\ref{Table3.5}. As with Example~\ref{example3.3},
the error for $\beta=k$, $k \in \mathbb{N}$, is equal to zero when we consider $N\geq k$.
\begin{figure}[!htb]
\begin{center}
\includegraphics[scale=1.0]{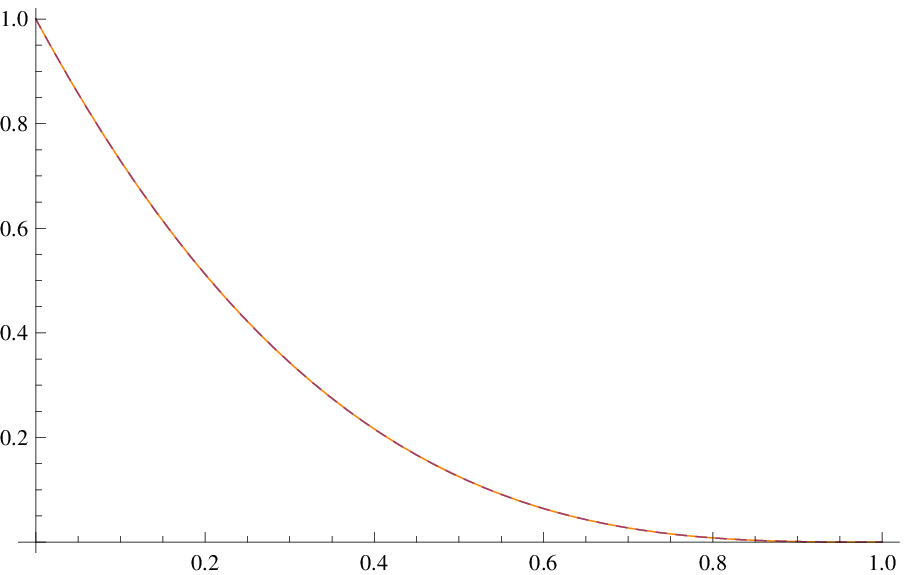}
\begin{minipage}{0.6\textwidth}
\caption{{\small Exact solution to problem \eqref{eq:prb:ex4}--\eqref{eq:prb:ex4:bc}
of Example~\ref{example4} with $\alpha=0.39$ and $\beta=3$ (the solid line)
versus our approximate solution with $N=3$ (the dashed line).}}
\label{fig6}
\end{minipage}
\end{center}
\end{figure}
\begin{table}[ht]
\begin{center}
\begin{minipage}{0.55\textwidth}
\caption{{\small Errors for problem \eqref{eq:prb:ex4}--\eqref{eq:prb:ex4:bc}
of Example~\ref{example4} when computed with \eqref{error} for different values
of $(\alpha,\beta)$ and $N=3,4,5$.}}
\label{Table3.5}
\end{minipage}
\begin{tabular}{|c|c|c|c|c|c|c|}
\hline
$(\alpha,\beta)$ & $N=3$&$N=4$&$N=5$\\\hline
$(0.39,1.5)$ & $6.1\times10^{-6}$&$8.1\times10^{-7}$&$1.7\times10^{-8}$\\
$(0.39,2.5)$ & $1.3\times10^{-6}$&$4.5\times10^{-8}$&$3.7\times10^{-9}$ \\
$(0.59,1.5)$ & $6.2\times10^{-6}$&$8.1\times10^{-7}$&$4.5\times10^{-9}$ \\
$(0.59,2.5)$ & $1.3\times10^{-6}$&$4.5\times10^{-8}$&$4.5\times10^{-9}$\\\hline
\end{tabular}
\end{center}
\end{table}
\end{example}


We finish by applying our method to three fractional optimal control problems.

\begin{example} \cite{sal15,sal17}
\label{exam5}
Consider the following fractional optimal control problem (FOCP):
\begin{equation}
\label{eq:J:OC:ex5}
\text{minimize}\quad J\{x,u\}=\int_0^1(tu(t)-(\alpha +2)x(t))^2 dt,
\end{equation}
subject to the dynamical fractional control system
\begin{equation}
\label{eq:J:OC:ex5:cs}
x'(t)+{_0^CD_t^\alpha} x(t)=u(t)+t^2
\end{equation}
and the boundary conditions
\begin{equation}
\label{eq:J:OC:ex5:bc}
x(0)=0, \quad x(1)=\frac{2}{\Gamma(3+\alpha)}.
\end{equation}
The exact solution is given by
\begin{equation*}
(x(t),u(t))=\left(\frac{2t^{\alpha+2}}{\Gamma(\alpha +3)},
\frac{2t^{\alpha+1}}{\Gamma(\alpha +2)}\right).
\end{equation*}
To solve this problem with our method, we take
\begin{equation*}
x_N(t)=\left(\frac{2}{\Gamma(3+\alpha)}
-\sum_{i=1}^N c_i\right)t+\sum_{i=1}^N c_i t^i.
\end{equation*}
The results for $\alpha=0.5$ and $N=3,4$ are plotted in Figure~\ref{fig7}.
Table~\ref{Table3.9} shows the errors for $\alpha=0.5$ and $N=2,3,4$.
\begin{figure}[!htb]
\begin{center}
\subfloat[$x(t)$, $N=3$]{\label{label:fig1}
\includegraphics[scale=0.77]{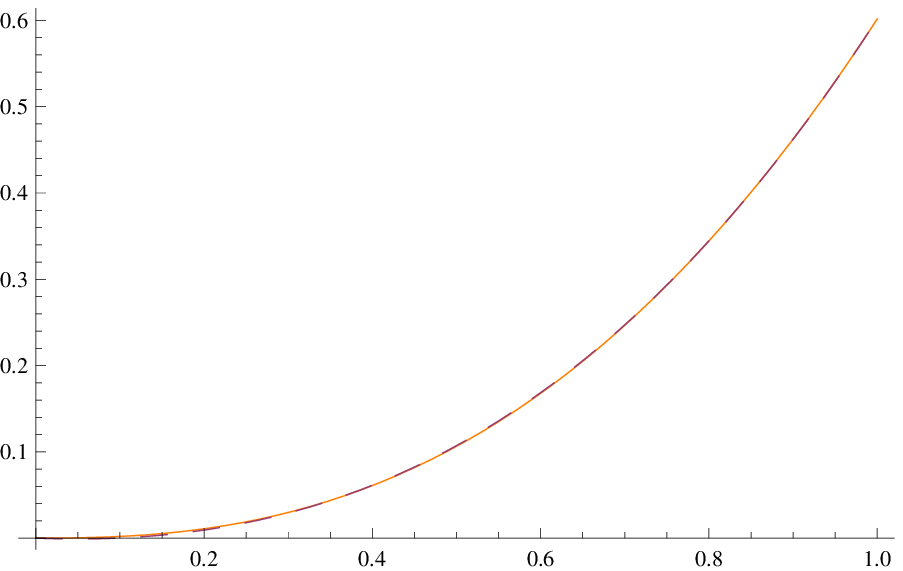}}
\qquad
\subfloat[$u(t)$, $N=3$]{\label{label:fig2}
\includegraphics[scale=0.77]{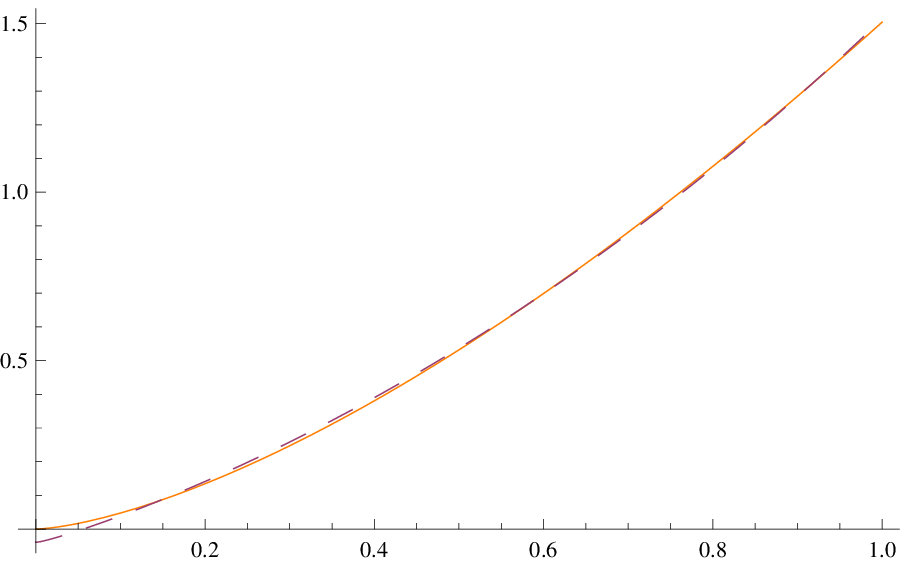}}\\
\subfloat[$x(t)$, $N=4$]{\label{label:fig3}
\includegraphics[scale=0.77]{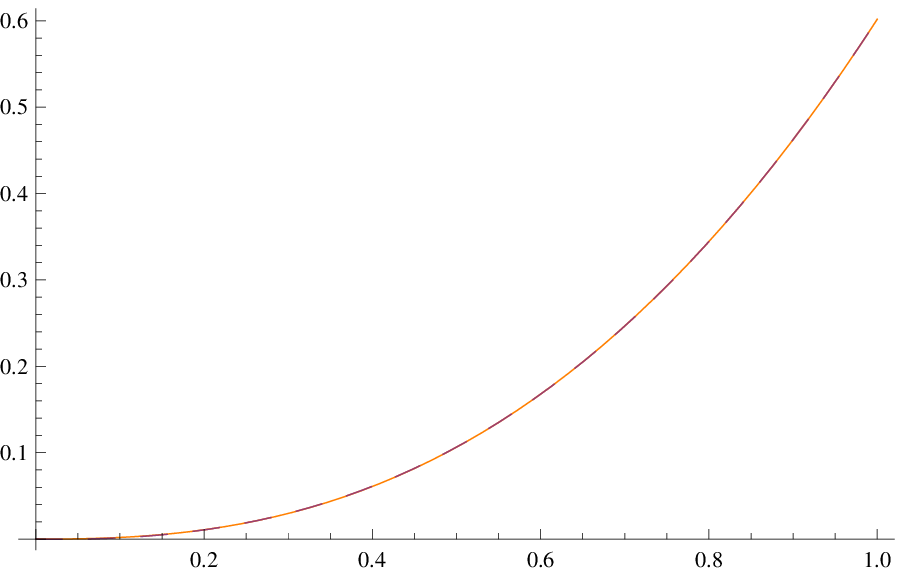}}
\qquad
\subfloat[$u(t)$, $N=4$]{\label{label:fig4}
\includegraphics[scale=0.77]{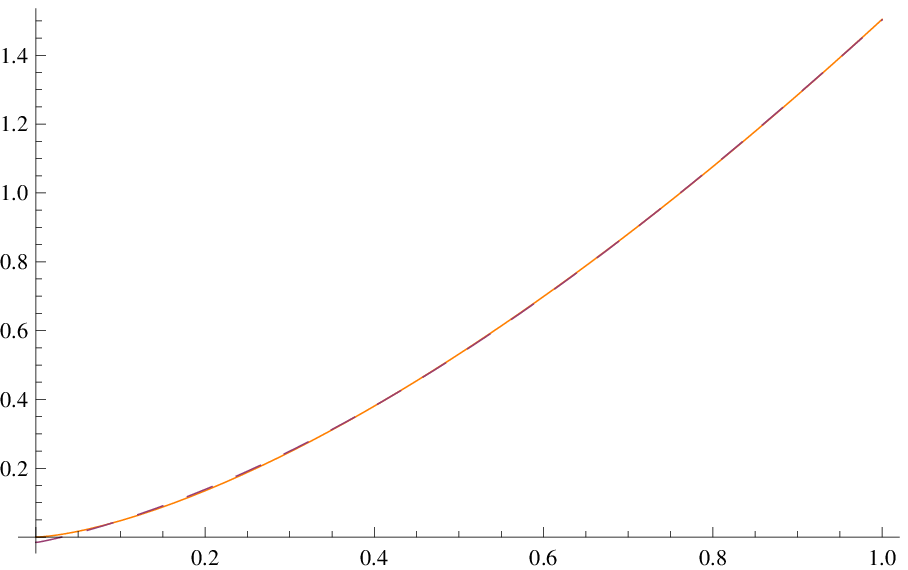}}\\
\begin{minipage}{0.95\textwidth}
\caption{{\small Exact solution $(x(t),u(t))$ to the FOCP \eqref{eq:J:OC:ex5}--\eqref{eq:J:OC:ex5:bc}
of Example~\ref{exam5} with $\alpha=0.5$ (the solid line) versus
our approximate solution with $N=3,4$ (the dashed line).}}
\label{fig7}
\end{minipage}
\end{center}
\end{figure}
\begin{table}[ht]
\begin{center}
\begin{minipage}{0.5\textwidth}
\caption{{\small Errors for problem \eqref{eq:J:OC:ex5}--\eqref{eq:J:OC:ex5:bc}
of Example~\ref{exam5} when computed with \eqref{error} for $\alpha=0.5$ and $N=2,3,4$.}}
\label{Table3.9}
\end{minipage}
\begin{tabular}{|c|c|c|c|c|c|c|}
\hline
Errors & $N=2$&$N=3$&$N=4$\\\hline
$Error\{x,x_N\}$ & $1.2\times10^{-4}$&$7.1\times10^{-7}$&$2.8\times10^{-8}$\\
$Error\{u,u_N\}$ & $7.1\times10^{-3}$&$9.6\times10^{-5}$&$7.9\times10^{-6}$ \\\hline
\end{tabular}
\end{center}
\end{table}
\end{example}


\begin{example} \cite{sal17}
\label{exam6}
Consider now the following FOCP:
\begin{equation}
\label{eq:J:OC:ex6}
\text{minimize}\quad J\{x,u\}=\int_0^1(u(t)-x(t))^2 dt,
\end{equation}
subject to the fractional dynamical control system
\begin{equation}
\label{eq:J:OC:ex6:cs}
x'(t)+{_0^CD_t^\alpha} x(t)=u(t)-x(t)+\frac{6t^{\alpha +2}}{\Gamma(\alpha +3)}+t^3
\end{equation}
and the boundary conditions
\begin{equation}
\label{eq:J:OC:ex6:bc}
x(0)=0, \quad x(1)=\frac{6}{\Gamma(4+\alpha)}.
\end{equation}
The exact solution to \eqref{eq:J:OC:ex6}--\eqref{eq:J:OC:ex6:bc} is given by
\begin{equation*}
(x(t),u(t))=\left(\frac{6t^{\alpha+3}}{\Gamma(\alpha +4)},
\frac{6t^{\alpha+3}}{\Gamma(\alpha +4)}\right).
\end{equation*}
To solve problem \eqref{eq:J:OC:ex6}--\eqref{eq:J:OC:ex6:bc} with our method, we take
\begin{equation*}
x_N(t)=\left(\frac{6}{\Gamma(4+\alpha)}
-\sum_{i=1}^N c_i\right)t+\sum_{i=1}^N c_i t^i.
\end{equation*}
The results for $\alpha=0.5$ and $N=3,4$ are plotted in Figure~\ref{fig8}.
The errors \eqref{error} for $\alpha=0.5$ and $N=2,3,4$
are shown in Table~\ref{Table3.10}.
\begin{figure}[!htb]
\begin{center}
\subfloat[$x(t)$, $N=3$]{\label{label:fig15}
\includegraphics[scale=0.77]{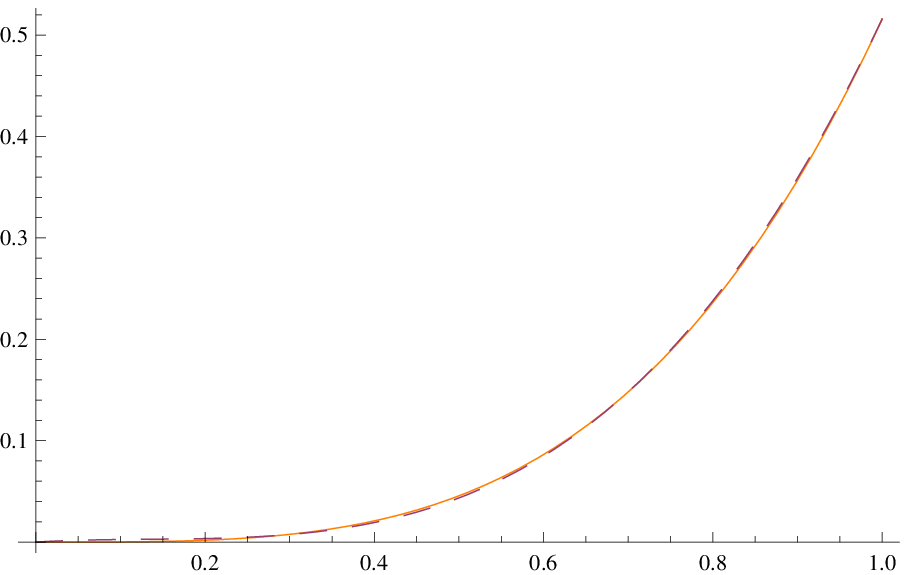}}
\qquad
\subfloat[$u(t)$, $N=3$]{\label{label:fig25}
\includegraphics[scale=0.77]{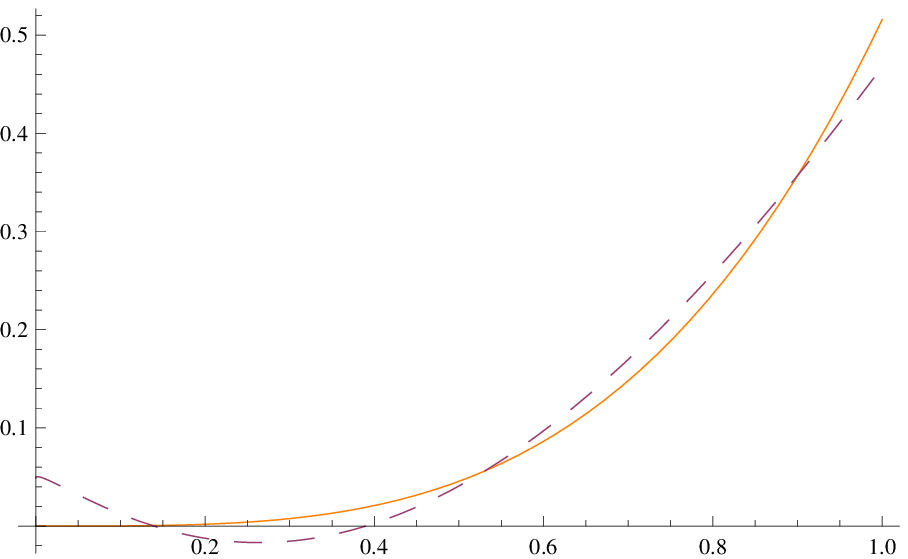}}\\
\subfloat[$x(t)$, $N=4$]{\label{label:fig35}
\includegraphics[scale=0.77]{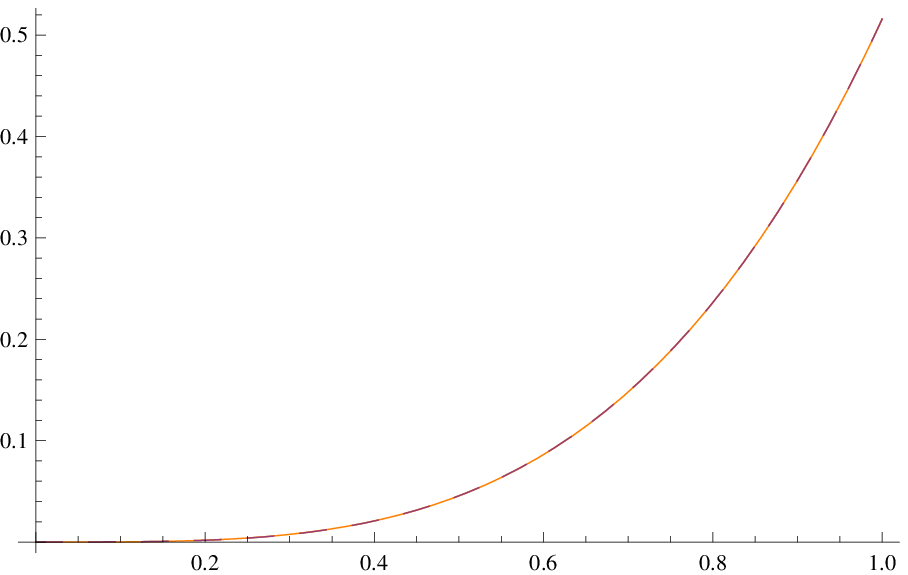}}
\qquad
\subfloat[$u(t)$, $N=4$]{\label{label:fig45}
\includegraphics[scale=0.77]{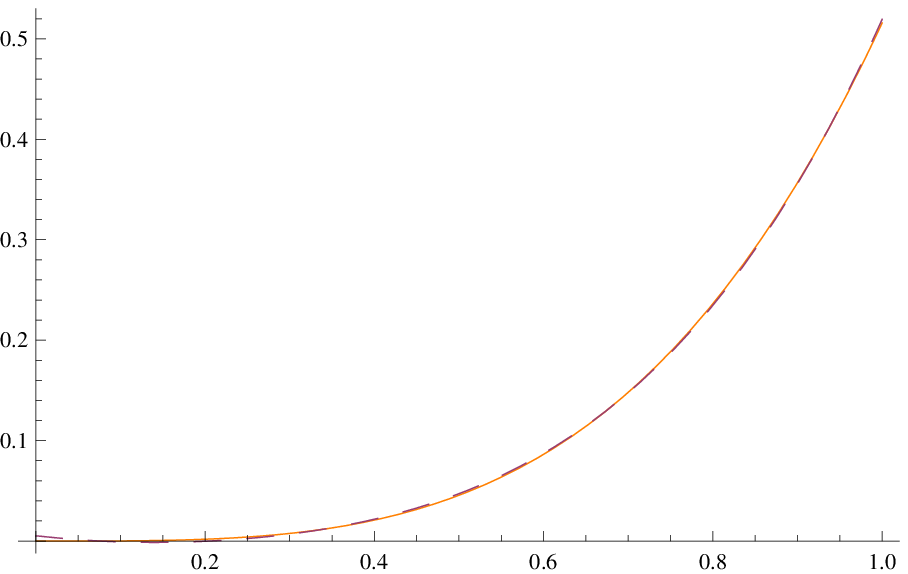}}\\
\begin{minipage}{0.95\textwidth}
\caption{{\small Exact solution $(x(t),u(t))$
to the FOCP \eqref{eq:J:OC:ex6}--\eqref{eq:J:OC:ex6:bc}
of Example~\ref{exam6} with $\alpha=0.5$ (the solid line) versus
our approximate solution with $N=3,4$ (the dashed line).}}
\label{fig8}
\end{minipage}
\end{center}
\end{figure}
\begin{table}[ht]
\begin{center}
\begin{minipage}{0.5\textwidth}
\caption{{\small Errors for problem \eqref{eq:J:OC:ex6}--\eqref{eq:J:OC:ex6:bc}
of Example~\ref{exam6} when computed with \eqref{error} for $\alpha=0.5$ and $N=2,3,4$.}}
\label{Table3.10}
\end{minipage}
\begin{tabular}{|c|c|c|c|c|c|c|}
\hline
Errors & $N=2$&$N=3$&$N=4$\\\hline
$Error\{x,x_N\}$ & $7.3\times 10^{-4}$&$2.5\times10^{-6}$&$8.4\times10^{-9}$\\
$Error\{u,u_N\}$ & $5.2\times 10^{-2}$&$3.9\times10^{-4}$&$2.2\times10^{-6}$ \\\hline
\end{tabular}
\end{center}
\end{table}
\end{example}

In \cite{MyID:294} the authors propose a direct method
for solving fractional optimal control problems, which involves
approximating the initial fractional order problem
by a new one with integer order derivatives only.
The latter problem is then discretized, by application of finite differences,
and solved numerically. Our method is simpler and does not involve
the solution of a nonlinear programming problem through AMPL and IPOPT.
Moreover, as we see next, it provides a much better result
when compared with the example given in \cite{MyID:294}.

\begin{example}\cite{MyID:294}
\label{exam11}
Consider the FOCP of \cite{MyID:294}:
\begin{equation}
\label{eq:exam11}
\mathrm{minimize}
\quad J\{x,u\}= \int_0^1 (u^2(t)- 4x(t))^2 dt,
\end{equation}
subject to the control system
\begin{equation}
\label{eq:exam12}
x'(t)+{_0^CD_t^{0.5}} x(t)=u(t)+ \frac{2}{\Gamma(2.5)} t^{1.5},
\quad t\in [0,1],
\end{equation}
and the boundary conditions
\begin{equation}
\label{eq:exam13}
x(0)=0, \quad \text{and} \quad x(1)=1.
\end{equation}
The exact solution to \eqref{eq:exam11}--\eqref{eq:exam13} is given by
\begin{equation*}
(x(t),u(t))= (t^2,2t)
\end{equation*}
(see \cite{MyID:294}). To solve problem \eqref{eq:exam11}--\eqref{eq:exam13}
with the method of this paper, we take
\begin{equation*}
x_N(t)= \left(1-\sum_{i=1}^N c_i\right)t+ \sum_{i=1}^N c_it^i.
\end{equation*}
In this case our method gives the exact solution
$x(t) = t^2$ for $N=2$. This is in contrast with
the results in \cite{MyID:294}, for which the exact
solution is not found. In fact, our method provides here
better results with just 2 steps ($c_1 = 0$ and $c_2 = 1$)
than the method introduced in \cite{MyID:294} with 100 steps.
\end{example}

\begin{remark}
According to relations \eqref{eq:yN:L} and \eqref{eq:yN:R},
accuracy depends on $N$. When the order $\alpha$
of the derivative changes, then the problem
under consideration also changes.
We were not able to find any general relation
between $\alpha$ and the accuracy, that is,
a general pattern on how $N$ changes with $\alpha$,
for a fixed precision. This seems to depend
on the particular situation at hand.
\end{remark}


\section{Conclusions}
\label{sec:conc}

We introduced a numerical method, based on Ritz's direct method,
to solve problems of the fractional calculus
of variations and fractional optimal control.
The idea of this approach is simple: by using some known basis functions,
we construct an approximate solution, which is a linear combination
of the basis functions, carrying out a finite-dimensional minimization
among such linear combinations and approximating the exact solution of the fractional
optimal control problem. From the simulation results, the proposed approach
is surprisingly accurate, and the approximate solution is nearly a perfect match
with the exact solution, which is superior to the results from previous
numerical methods available in the literature.


\begin{acknowledgements}
This work is part of first author's PhD project.
It was partially supported by Islamic Azad University, Tehran, Iran;
and CIDMA-FCT, Portugal, within project UID/MAT/04106/2013.
Jahanshahi was also supported by a scholarship from the Ministry of Science,
Research and Technology of the Islamic Republic of Iran,
to visit the University of Aveiro, Portugal, and work with Professor Torres.
The hospitality and the excellent working conditions
at the University of Aveiro are here gratefully acknowledged.
The authors are indebted to an anonymous referee for a
careful reading of the original manuscript and for providing
several suggestions, questions, and remarks. They are also
grateful to the Editor-in-Chief, Professor Giannessi,
and Ryan Loxton, for English improvements.
\end{acknowledgements}




\begin{thebibliography}{99}

\bibitem{MR2218073}
Kilbas, A.A., Srivastava, H.M., Trujillo, J.J.:
Theory and applications of fractional differential equations.
North-Holland Mathematics Studies, 204,
Elsevier, Amsterdam (2006)

\bibitem{MR1347689}
Samko, S.G., Kilbas, A.A., Marichev, O.I.:
Fractional integrals and derivatives.
Translated from the 1987 Russian original.
Gordon and Breach, Yverdon (1993)

\bibitem{MR3181071}
Val\'erio, D., Tenreiro Machado, J., Kiryakova, V.:
Some pioneers of the applications of fractional calculus.
Fract. Calc. Appl. Anal. 17:2, 552--578 (2014)

\bibitem{MR3224387}
de Oliveira, E.C., Tenreiro Machado, J.A.:
A review of definitions for fractional derivatives and integral.
Math. Probl. Eng. 2014:238459, 6~pp (2014)

\bibitem{MR2960307}
Ortigueira, M.D., Trujillo, J.J.:
A unified approach to fractional derivatives.
Commun. Nonlinear Sci. Numer. Simul. 17:12, 5151--5157 (2012)

\bibitem{MR2768178}
Ortigueira, M.D.:
Fractional calculus for scientists and engineers.
Lecture Notes in Electrical Engineering, 84,
Springer, Dordrecht (2011)

\bibitem{MR3188372}
Tenreiro Machado, J.A., Baleanu, D., Chen, W., Sabatier, J.:
New trends in fractional dynamics.
J. Vib. Control 20:7, 963 (2014)

\bibitem{sal34}
Almeida, R., Pooseh, S., Torres, D.F.M.:
Computational methods in the fractional calculus of variations.
Imp. Coll. Press, London (2015)

\bibitem{book:adv:FCV}
Malinowska, A.B., Odzijewicz, T., Torres, D.F.M.:
Advanced methods in the fractional calculus of variations.
Springer Briefs in Applied Sciences and Technology, Springer, Cham (2015)

\bibitem{MR2984893}
Malinowska, A.B., Torres, D.F.M.:
Introduction to the fractional calculus of variations.
Imp. Coll. Press, London (2012)

\bibitem{sal1}
Riewe, F.:
Nonconservative Lagrangian and Hamiltonian mechanics.
Phys. Rev. E (3) 53:2, 1890--1899 (1996)

\bibitem{sal2}
Riewe, F.:
Mechanics with fractional derivatives.
Phys. Rev. E (3) 55:3, part B, 3581--3592 (1997)

\bibitem{sal5}
Agrawal, O.P.:
Fractional variational calculus in terms of Riesz fractional derivatives.
J. Phys. A 40:24, 6287--6303 (2007)

\bibitem{sal3}
Almeida, R., Torres, D.F.M.:
Leitmann's direct method for fractional optimization problems.
Appl. Math. Comput. 217:3, 956--962 (2010)
{\tt arXiv:1003.3088}

\bibitem{sal4}
Almeida, R., Torres, D.F.M.:
Necessary and sufficient conditions for the fractional
calculus of variations with Caputo derivatives.
Commun. Nonlinear Sci. Numer. Simul. 16:3, 1490--1500 (2011)
{\tt arXiv:1007.2937}

\bibitem{MR3103208}
Atanackovi\'c, T.M., Janev, M., Konjik, S., Pilipovi\'{c},  S., Zorica, D.:
Expansion formula for fractional derivatives in variational problems.
J. Math. Anal. Appl. 409:2, 911--924 (2014)

\bibitem{MR3162654}
Baleanu, D., Garra, R., Petras, I.:
A fractional variational approach to the fractional Basset-type equation.
Rep. Math. Phys. 72:1, 57--64 (2013)

\bibitem{MR3200762}
Bourdin, L., Odzijewicz, T., Torres, D.F.M.:
Existence of minimizers for generalized Lagrangian functionals and a necessary
optimality condition---application to fractional variational problems.
Differential Integral Equations 27:7-8, 743--766 (2014)
{\tt arXiv:1403.3937}

\bibitem{MR3221831}
Odzijewicz, T., Torres, D.F.M.:
The generalized fractional calculus of variations.
Southeast Asian Bull. Math. 38:1, 93--117 (2014)
{\tt arXiv:1401.7291}

\bibitem{sal6}
Almeida, R., Khosravian-Arab, H., Shamsi, M.:
A generalized fractional variational problem depending on indefinite integrals:
Euler-Lagrange equation and numerical solution.
J. Vib. Control 19:14, 2177--2186 (2013)

\bibitem{sal7}
Blaszczyk, T., Ciesielski, M.:
Numerical solution of fractional Sturm-Liouville equation in integral form.
Fract. Calc. Appl. Anal. 17:2, 307--320 (2014)

\bibitem{MyID:294}
Almeida, R., Torres, D.F.M.:
A discrete method to solve fractional optimal control problems.
Nonlinear Dynam. 80:4, 1811--1816 (2015)
{\tt arXiv:1403.5060}

\bibitem{sal8}
Pooseh, S., Almeida, R., Torres, D.F.M.:
Numerical approximations of fractional derivatives with applications.
Asian J. Control 15:3, 698--712 (2013)
{\tt arXiv:1208.2588}

\bibitem{sal9}
Dehghan, M., Hamedi, E.-A., Khosravian-Arab, H.:
A numerical scheme for the solution of a class of fractional variational
and optimal control problems using the modified Jacobi polynomials.
J. Vib. Control, in press.
DOI:10.1177/1077546314543727

\bibitem{sal16}
Caputo, M.:
Linear models of dissipation whose $Q$ is almost frequency independent.
II, Fract. Calc. Appl. Anal. 11:1, 4--14 (2008)

\bibitem{sal14}
Agrawal, O.P.:
A general formulation and solution scheme for fractional optimal control problems.
Nonlinear Dynam. 38:1-4, 323--337 (2004)

\bibitem{MR2433010}
Frederico, G.S.F., Torres, D.F.M.:
Fractional conservation laws in optimal control theory.
Nonlinear Dynam. 53:3, 215--222 (2008)
{\tt arXiv:0711.0609}

\bibitem{MR2386201}
Frederico, G.S.F., Torres, D.F.M.:
Fractional optimal control in the sense of Caputo and the fractional Noether's theorem.
Int. Math. Forum 3:9-12, 479--493 (2008)
{\tt arXiv:0712.1844}

\bibitem{sal15}
Pooseh, S., Almeida, R., Torres, D.F.M.:
Fractional order optimal control problems with free terminal time.
J. Ind. Manag. Optim. 10:2, 363--381 (2014)
{\tt arXiv:1302.1717}

\bibitem{sal17}
Sweilam, N.H., Al-Ajami, T.M., Hoppe, R.H.W.:
Numerical solution of some types of fractional optimal control problems.
The Scientific World Journal 2013:306237, 9~pp (2013)

\end{thebibliography}
\end{document}